\documentclass[11pt]{article}

\usepackage{amsmath}
\usepackage{latexsym}
\usepackage{amssymb}
\usepackage{graphicx}
\numberwithin{equation}{section}

\newtheorem{theorem}{\bf Theorem}[section]
\newtheorem{definition}{\bf Definition}[section]
\newtheorem{lemma}{\bf Lemma}[section]

\def\udots{\mathinner{\mkern1mu\raise-1pt\vbox{\kern7pt\hbox{.}}\mkern2mu
    \raise2pt\hbox{.}\mkern2mu\raise5pt\hbox{.}\mkern1mu}}

\allowdisplaybreaks

\begin{document}
\begin{center}
{\Large \bf Elliptic Stochastic Partial Differential Equations with Two Reflecting Walls}
\end{center}

\begin{center}
Wen Yue, Tusheng Zhang
\end{center}
\begin{center}
{\scriptsize Department of Mathematics, University of Manchester, Oxford Road, Manchester M13 9PL, England, UK}
\end{center}

\begin{abstract}
In this article, we study elliptic stochastic partial differential equations with two reflecting walls $h^{1}$ and $h^{2}$, driven by multiplicative noise. The existence and uniqueness of the solutions are established.
\end{abstract}



\section{Introduction}

In this paper we will consider the following elliptic stochastic partial differential equations (SPDEs) with Dirichlet boundary condition on a bounded domain $D$ of $\mathbb{R}^{k}$, $k=1,2,3$.
\begin{equation}
-\Delta u(x)+f(x;u(x))=\eta(x)-\xi(x)+\sigma(x;u(x))\dot{W}(x),\ \ x\in D, \label{elliptic equation}
\end{equation}
where $\{\dot{W}(x),\ x\in D\}$ is a white noise in $D$. We are looking for a continuous random field ${u(x),\ x\in D}$ which is the solution of equation (\ref{elliptic equation}) satisfying $h^{1}(x)\leq u(x)\leq h^{2}(x)$, where $h^{1}$ and $h^{2}$ are given two walls. When $u(x)$ hits $h^{1}(x)$ or $h^{2}(x)$, the additional forces are added to prevent u from leaving $[h^{1},h^{2}]$. These forces are expressed by random measures $\xi$ and $\eta$ in equation(\ref{elliptic equation}) which play a similar role as the local time in the usual Skorokhod equation constructing Brownian motions with reflecting barriers. SPDEs with two reflecting walls can be used to model the evolution of random interfaces near two hard walls, see T. Funaki and S. Olla \cite{FO}. For nonlinear elliptic PDEs with measures as right side or boundary condition, we refer to Boccardo, Gallouet \cite{BG} and Rockner, Zegarlinski \cite{RZ}.\\

 For elliptic SPDEs without reflection, R. Buckdahn and E. Pardoux in \cite{BP} established the existence and uniqueness of the solutions of nonlinear elliptic stochastic partial differential equations  driven by additive noise. Based on this, elliptic SPDEs with reflection at zero driven by additive noise, have been studied by David Nualart and Samy Tindel in  \cite{NT}.\\

In our present paper, we will study the elliptic SPDEs with two reflecting walls driven by multiplicative noise. This is the first time to consider the case of  multiplicative noise. We will establish the existence and uniqueness of the solutions.
A similar problem for reflected  stochastic heat equations  has been studied by Nualart and Pardoux in \cite{NP}, Donati-Martin and Pardoux in \cite{DP}, Yang and Zhang in \cite{YZ1} and by Xu and  Zhang in \cite{XZ}. Our approaches were inspired by the ones in \cite{NP}, \cite{NT},  \cite{O} and \cite{XZ}.\\
\vskip 0.3cm
The rest of the paper is organized as follows. In Section 2, we lay down the framework of the paper. In Section 3, we study deterministic reflected elliptic PDEs and obtain some  a priori estimates. The main result is established in Section 4.

\section{Framework}

Let $D$ be an open bounded subset of $\mathbb{R}^{k}$, with $k\in \{1,2,3\}$. Consider a Gaussian family of random variables $\{ W=W(B),B\in \mathcal{B}(D)\}$, where $\mathcal{B}(D)$ is the Borel $\sigma$-field on $D$, defined in a complete probability space $(\Omega,\mathcal{F},P)$, such that $E(W(B))=0$ and
\begin{equation}
E(W(A)W(B))=|A\cap B|,
\end{equation}
where $|A\cap B|$ denotes the Lebesgue measure of the set $A\cap B$. We want to study a reflected nonlinear stochastic elliptic equation with Dirichlet condition driven by multiplicative noise:
\begin{eqnarray}
-\Delta u(x)+f(x,u(x))=\sigma(x;u(x)) \dot{W}(x), \label{elliptic equation without reflection}
\end{eqnarray}
where $x\in D$ while $h^{1}(x)\leq u(x)\leq h^{2}(x)$,
$\dot{W}(x)$ is the formal derivative of $W$ with respect to the Lebesgue measure and the symbol $\Delta$ denotes the Laplace operator on $L^{2}(D)$ . If $u(x)$ hits $h^{1}(x)$ or $h^{2}(x)$, additional forces are added in order to prevent $u$ from leaving $[h^{1},h^{2}]$. Such an effect will be expressed by adding extra(unknown) terms $\xi$ and $\eta$ in (\ref{elliptic equation without reflection}) which play a similar role as the local time in the usual Skorokhod equation constructing Brownian motions with reflecting boundaries.

$\mathcal{C}_{0}^{\infty}(D)$ denotes the set of infinitely differentiable functions on $D$ with compact supports. We will denote by $(\cdot,\cdot)$ the scalar product in $L^{2}(D)$, and by $||\cdot||_{\infty}$ the supremum norm on $D$. Let $f, \sigma:D\times \mathbb{R}\rightarrow \mathbb{R}$ be measurable functions. We will also denote by $f(u)$ the function $f(u)(x)=f(x,u(x))$,  $\sigma(u)$  the function $\sigma(u)(x)=\sigma(x,u(x))$.
We introduce the following hypotheses on f and $\sigma$:
\vskip 0.3cm
(F1) The function f is locally bounded, continuous and nondecreasing as a function of the second variable.
\vskip 0.3cm
($\Sigma$ 1) The function $\sigma$ is Lipschitz continuous:
\begin{eqnarray*}
|\sigma(x,z_{1})-\sigma(x,z_{2})|\leq C_{\sigma}|z_{1}-z_{2}|.
\end{eqnarray*}
\vskip 0.3cm
(H1) The walls $h^{i}(x), i=1,2$, are continuous functions satisfying $h^{1}(x)<h^{2}(x)$ for $ x\in D$ and
$h^{1}(x)\leq 0\leq h^{2}(x)$ for $ x\in \partial D$.
\vskip 0.3cm
The solution to Eq(\ref{elliptic equation}) will be a triplet $(u, \eta, \xi)$ such that $h^{1}(x)\leq u(x)\leq h^{2}(x)$ on $D$ which satisfies Eq(\ref{elliptic equation}) in the sense of distributions, and $\eta(dx)$, $\xi(dx)$ are random measures on $D$ which force the process $u$ to be in the interval $[h^{1},h^{2}]$. More precisely, a rigorous definition of the solution to Eq(\ref{elliptic equation}) is given as follows:

\begin{definition}
A triplet $(u,\eta,\xi)$ defined on a complete probability space $(\Omega, \mathcal{F},P)$ is a solution to the SPDE (\ref{elliptic equation}), denoted by $(0;f;\sigma;h^{1},h^{2})$, if\\
(i) $\{u(x),x\in D \}$ is a continuous random field on $D$ satisfying $h^{1}(x)\leq u(x)\leq h^{2}(x)$ and $ u_{|\partial D}=0$ a.s.\\
(ii) $\eta(dx)$ and $\xi(dx)$ are random measures on $D$ such that $\eta(K) <\infty $ and $\xi(K)<\infty$ for all compact subset $K\subset D$.\\
(iii) For all $\phi \in \mathcal{C}_{0}^{\infty}(D)$, we have
\begin{equation}
-(u,\Delta \phi)+(f(u),\phi)=\int_{D}\phi(x)\sigma(u)W(dx)+\int_{D}\phi(x)\eta({dx})-\int_{D}\phi(x)\xi(dx).\ \  P-a.s.
\end{equation}
(iv) $ \int_{D}(u(x)-h^{1}(x))\eta(dx)=\int_{D}(h^{2}(x)-u(x))\xi(dx)=0$.
\end{definition}

\section{Deterministic obstacle problem}
Let $h^{1}, h^{2}$ be as in Section 2, and $f$ satisfies $(F1)$. Let $v(x) \in C(D)$ with $v|_{\partial D}=0$. Consider a deterministic elliptic PDE with two reflecting walls:
\begin{equation}
  \left\{
   \begin{aligned}
-\Delta z+f(z+v)=\eta-\xi\\
h^{1}\leq z+v \leq h^{2}\\
z|_{\partial D}=0 .
    \end{aligned}
    \right. \label{deterministic eq1}
\end{equation}
Here is a precise definition of the solution of equation (\ref{deterministic eq1}).
\begin{definition}
A triplet $(z,\eta,\xi)$ is called a solution to the PDE (\ref{deterministic eq1}) if\\
(i) $z=z(x); x\in D$ is a continuous function satisfying $h^{1}(x)\leq z(x)+v(x) \leq h^{2}(x)$, $z|_{\partial D}=0$.\\
(ii) $\eta(dx)$ and $\xi(dx)$ are measures on $D$ such that $\eta(K) <\infty $ and $\xi(K)<\infty$ for all compact subset $K\subset D$.\\
(iii) For all $\phi \in \mathcal{C}_{0}^{\infty}(D)$ we have
\begin{equation}
-(z,\Delta \phi)+(f(z+v),\phi)=\int_{D}\phi(x)\eta({dx})-\int_{D}\phi(x)\xi(dx).
\end{equation}
(iv)$ \int_{D}(z(x)+v(x)-h^{1}(x))\eta(dx)=\int_{D}(h^{2}(x)-z(x)-v(x))\xi(dx)=0$.
\end{definition}

The following result is the existence and uniqueness of the solutions of the PDE with two reflecting walls (\ref{deterministic eq1}).

\begin{theorem}
Equation (\ref{deterministic eq1}) admits a unique solution ($z,\eta,\xi$).
\end{theorem}

We first consider the problem of a single reflecting barrier, denoted by $(0;f;h^{1})$:
\begin{equation}\label{single barrier}
  \left\{
   \begin{aligned}
-\Delta z+f(z+v)=\eta(x)\\
 z+v \geq h^{1}\\
z|_{\partial D}=0 \\
\int_{D}(z+v-h^1)\eta(dx)=0,
    \end{aligned}
    \right.
\end{equation}
where the coefficient $f$ satisfies $(F1)$ and $h^1$ satisfies $(H1)$ in Section 2.\\

In the next lemma, we give the existence and uniqueness of the solution of $(0;f;h^1)$, and it follows from Theorem 2.2 in David Nualart and Samy Tindle \cite{NT} using similar methods.

\begin{lemma}
Let $v$ be a continuous function on $\bar{D}$ such that $v|_{\partial D} =0$. There exists a unique pair $(z,\eta)$ such that:\\
(i) $z$ is a continuous function on $\bar{D}$ such that $z|_{\partial D}=0$ and $z+v \geq h^{1} $.\\
(ii) $\eta$ is a measure on  on $D$ such that $\eta(K)< \infty$ for any compact set $K\subset D$.\\
(iii) For every $\phi \in \mathcal{C}_k^\infty(\mathcal{D})$, we have
\begin{eqnarray*}
-(z,\Delta \phi)+(f(z+v),\phi)=\int_{D}\phi(x)\eta(dx).
\end{eqnarray*}
(iv)$\int_{D}(z(x)+v(x)-h^1(x))\eta(dx)=0.$\\

\end{lemma}

Theorem 2.2 from David Nualart and Samy Tindel:\\
Let $v$ be a continuous function on $\bar{D}$ such that $v|_{\partial D=0}$. There exist a unique pair $(z,\eta)$ such that:\\
(i) $z$ is a continuous function on $\bar{D}$ such that $z|_{\partial D}=0$ and $z \geq -v$.\\
(ii) $\eta$ is a measure on $D$ such that $\eta(K)< \infty$ for any compact set $K\subset D$.\\
(iii) For every $\phi \in \mathcal{C}_k^\infty(\mathcal{D})$, we have
\begin{eqnarray*}
-(z,\Delta \phi)+(f(z+v),\phi)=\int_{D}\phi(x)\eta(dx).
\end{eqnarray*}
(iv)$\int_{D}(z(x)+v(x))\eta(dx)=0.$\\

Next lemma is a comparison theorem for the PDE with reflection.

\begin{lemma}
(comparison)\\
Let $(z_{1},\eta_{1})$ and $(z_{2},\eta_{2})$ be solutions to single reflection problems $(0;f_{1},h_{1})$ and $(0;f_{2},h_{2})$ respectively as in (\ref{single barrier}). If $f_{1}\leq f_{2}$, and $h_{1}\geq h_{2}$, for every $x\in D$, then we have $z_{1}(x)\geq z_{2}(x)$.
\end{lemma}

\begin{proof}
Let $z_{1}^{\epsilon}$ and $z_{2}^{\epsilon}$ be the solutions of the following PDEs:
\begin{equation}
  \left\{
   \begin{aligned}
-\Delta z_{1}^{\epsilon}(x)+f_{1}(z_{1}^{\epsilon}+v)(x)=\frac{1}{\epsilon}(z_{1}^{\epsilon}+v-h_1)^{-}(x) \\
z_{1}^{\epsilon}|_{\partial D}=0.
   \end{aligned}
  \right.  \label{1}
\end{equation}

\begin{equation}
  \left\{
   \begin{aligned}
-\Delta z_{2}^{\epsilon}(x)+f_{2}(z_{2}^{\epsilon}+v)(x)=\frac{1}{\epsilon}(z_{2}^{\epsilon}+v-h_2)^{-}(x) \\
z_{2}^{\epsilon}|_{\partial D}=0.
   \end{aligned}
  \right.  \label{2}
\end{equation}
According to \cite{NT}, $z_{1}^{\epsilon} \rightarrow z_{1}$ and $z_{2}^{\epsilon} \rightarrow z_{2}$ uniformly on $\bar{D}$ as $\epsilon \rightarrow 0$.\\
Let $\psi =z_{2}^{\epsilon}-z_{1}^{\epsilon}$, then
\begin{equation}
  \left\{
   \begin{aligned}
-\Delta \psi+f_{2}(z_{2}^{\epsilon}+v)-f_{1}(z_{1}^{\epsilon}+v)=\frac{1}{\epsilon}[(z_{2}^{\epsilon}+v-h_2)^{-}-
(z_{1}^{\epsilon}+v-h_1)^{-}] \\
\psi |_{\partial D}=0.
   \end{aligned}
  \right.  \label{3}
\end{equation}

Multiplying  (\ref{3}) by $\psi^{+}$, we obtain
\begin{eqnarray}
(-\Delta \psi, \psi^{+})+(f_{2}(z_{2}^{\epsilon}+v)-f_{1}(z_{1}^{\epsilon}+v),\psi^{+})
=\frac{1}{\epsilon}([(z_{2}^{\epsilon}+v-h_{2})^{-}-(z_{1}^{\epsilon}+v-h_{1})^{-}],\psi^{+}) \label{difference}
\end{eqnarray}

Note that,
\begin{eqnarray}
-(\Delta\psi,\psi^{+})=(\triangledown \psi,\triangledown \psi^{+})=(\triangledown \psi^{+},\triangledown \psi^{+})=||\triangledown \psi^{+}||^{2}_{L^{2}(D)} \geq 0. \label{left1}
\end{eqnarray}

If $\psi^{+}(x)\neq 0$, we have $z_{2}^{\epsilon}(x)>z_{1}^{\epsilon}(x)$.
Because $f_{2}$ is increasing and $f_{1}\leq f_{2}$, we also have
\begin{eqnarray}
(f_{2}(z_{2}^{\epsilon}+v)-f_{1}(z_{1}^{\epsilon}+v),\psi^{+})\geq 0. \label{left2}
\end{eqnarray}

Since $h_{1}\geq h_{2}$, we  have $z_{2}^{\epsilon}(x)+v(x)-h_{2}(x)\geq z_{1}^{\epsilon}(x)+v(x)-h^{1}(x)$ and then
\begin{eqnarray}
\frac{1}{\epsilon}([(z_{2}^{\epsilon}(x)+v(x)-h_{2}(x))^{-}-(z_{1}^{\epsilon}(x)+v(x)-h_{1}(x))^{-}],\psi^{+}(x))\leq 0.\label{right1}
\end{eqnarray}

Thus it follows from (\ref{difference}),(\ref{left1}),(\ref{left2}) and (\ref{right1}) that:
$$
||\triangledown \psi^{+}||^{2}_{L^{2}(D)} = 0.
$$
Hence, by the boundary condition $\psi^{+}|_{\partial D}=0$, we get $\psi^{+}=0$ and then $z_{2}^{\epsilon}\leq z_{1}^{\epsilon}$, for every $\epsilon >0$.
Hence, the lemma follows immediately by taking $\epsilon \rightarrow 0$. $\Box$
\end{proof}

\begin{lemma}
Let $v$ and $\hat{v}$ be given continuous functions and let $z^{\epsilon,\delta}$ be a unique solution to the following deterministic PDE:
\begin{equation}
  \left\{
   \begin{aligned}
-\Delta z^{\epsilon,\delta}(x)+f(z^{\epsilon,\delta}+v)(x)=\frac{1}{\delta}(z^{\epsilon,\delta}(x)+v(x)-h^{1}(x))^{-}-\frac{1}{\epsilon}(z^{\epsilon,\delta}(x)+v(x)-h^{2}(x))^{+},\\
z^{\epsilon,\delta}|_{\partial D}=0.
   \end{aligned}
  \right.
\end{equation}
We also denote by $\hat{z}^{\epsilon,\delta}$ the solution to the above PDE replacing $v$ by $\hat{v}$. Then we have, $||z^{\epsilon,\delta}-\hat{z}^{\epsilon,\delta}||_{\infty}\leq ||v-\hat{v}||_{\infty}$, where $||w||_{\infty}=\sup_{x\in D}|w(x)|$.
\end{lemma}

\begin{proof}
Define $w(x)=z^{\epsilon,\delta}(x)-\hat{z}^{\epsilon,\delta}(x)-l$, where $l=||v-\hat{v}||_{\infty}$.\\
Then, $w$ satisfies the following PDE:
\begin{eqnarray}
    -\Delta w+f(z^{\epsilon,\delta}+v)-f(\hat{z}^{\epsilon,\delta}+\hat{v}) \nonumber
&=&\frac{1}{\delta}[(z^{\epsilon,\delta}+v-h^{1})^{-}-(\hat{z}^{\epsilon,\delta}+\hat{v}-h^{1})^{-}]\\
  &&-\frac{1}{\epsilon}[(z^{\epsilon,\delta}+v-h^{2})^{+}-(\hat{z}^{\epsilon,\delta}+\hat{v}-h^{2})^{+}] \label{1 in lemma 2}
\end{eqnarray}
Set
$$F_{\epsilon,\delta}(u)=f(u)-\frac{1}{\delta}(u-h^{1})^{-}+\frac{1}{\epsilon}(u-h^{2})^{+}$$
Now we note that, if $w^+(x)>0$, we have  $z^{\epsilon,\delta}(x)+v(x) > \hat{z}^{\epsilon,\delta}(x)+\hat{v}(x)$ and hence
\begin{eqnarray}
  \left\{
   \begin{aligned}
f(z^{\epsilon,\delta}+v)(x)\geq f(\hat{z}^{\epsilon,\delta}+\hat{v})(x)\\
\frac{1}{\delta}[(z^{\epsilon,\delta}+v-h^{1})^{-}(x)-(\hat{z}^{\epsilon,\delta}+\hat{v}-h^{1})^{-}(x)] \leq 0\\
\frac{1}{\epsilon}[(z^{\epsilon,\delta}+v-h^{2})^{+}(x)-(\hat{z}^{\epsilon,\delta}+\hat{v}-h^{2})^{+}(x)] \geq 0,
\label{elements in difference}
\end{aligned}
  \right.
\end{eqnarray}
Consequently, on the set $\{x\in D; w^+(x)>0\}$, we have
\begin{equation}\label{001}
F_{\epsilon,\delta}(z^{\epsilon,\delta}+v)(x)-F_{\epsilon,\delta}(\hat{z}^{\epsilon,\delta}+\hat{v})(x)\geq 0
\end{equation}
On the other hand,
multiplying (\ref{1 in lemma 2}) by $w^{+}$, we obtain:
\begin{eqnarray}
 -(\Delta w,w^{+})+(F_{\epsilon,\delta}(z^{\epsilon,\delta}+v)-F_{\epsilon,\delta}(\hat{z}^{\epsilon,\delta}+\hat{v}),w^{+})\nonumber &=&0 \label{difference's inner product}
\end{eqnarray}
Because $$-(\Delta w,w^{+})=||\triangledown w^{+}||_{L^{2}(D)}^{2}\geq 0,$$
it follows from (\ref{difference's inner product}) that
$$||\triangledown w^{+}||_{L^{2}(D)}^{2}=0$$
and
$$(F_{\epsilon,\delta}(z^{\epsilon,\delta}+v)-F_{\epsilon,\delta}(\hat{z}^{\epsilon,\delta}+\hat{v}),w^{+})\nonumber =0$$
Taking into account the fact $w^+=0$ on $\partial D$, we deduce  $w^{+}=0$. Hence $z^{\epsilon,\delta}-\hat{z}^{\epsilon,\delta}\leq l$.
Interchanging the role of $z^{\epsilon,\delta}$ and $\hat{z}^{\epsilon,\delta}$, we prove the lemma.
 $\Box$
\end{proof}\\

The next lemma is a straight consequence of the above lemma.

\begin{lemma}
Let $v$ and $\hat{v}$ be given continuous functions and let $(z^{\epsilon},\eta^{\epsilon})$ and $(\hat{z}^{\epsilon},\hat{\eta}^{\epsilon})$ be the unique solutions to single reflection problems $(0; f+\frac{(\cdot+v-h^{2})^{+}}{\epsilon};h^{1})$ and $(0; f+\frac{(\cdot+\hat{v}-h^{2})^{+}}{\epsilon};h^{1})$, respectively. Then we have $||z^{\epsilon}-\hat{z}^{\epsilon}||_{\infty}\leq ||v-\hat{v}||_{\infty}$.
\end{lemma}

\noindent Proof of Theorem 3.1:\\
Denote by $z^{\epsilon}$ the solution of the following single barrier problem:
\begin{equation}
  \left\{
   \begin{aligned}
-\Delta z^{\epsilon}+f(z^{\epsilon}+v)+\frac{1}{\epsilon}(z^{\epsilon}+v-h^{2})^{+}=\eta^{\epsilon}\\
z^{\epsilon}+v\geq h^{1}\\
\int_{D}(z^{\epsilon}+v-h^{1})\eta^{\epsilon}(dx)=0,
\end{aligned}
  \right.
\end{equation}
By the construction in  \cite{NT}, it is known that $\eta^{\epsilon}(dx)=\lim_{\delta \rightarrow 0}\frac{(z^{\epsilon,\delta}+v-h^{1})^{-}}{\delta}(dx)$ and it means that the measure $\frac{(z^{\epsilon,\delta}+v-h^{1})^{-}}{\delta}(dx)$ converges to $\eta^\epsilon(dx)$ in the sense of distribution on $D$.
According to lemma 3.2(comparison):
$z^{\epsilon}(x)$ is decreasing  as $\epsilon \downarrow 0$. Since $z^{\epsilon}(x)\geq h^{1}(x)-v(x)$,
$z^{\epsilon}(x)$ converge to some function $z(x$) as $\epsilon \rightarrow 0$. Using similar arguments as in the proof of Lemma 3.2 in \cite{NT}, we can show that the function $z(x)$ is also continuous.
\\
Next we prove that $z(x)$ is a solution of the reflected PDE with two reflected walls
  \begin{eqnarray}
    \left\{
   \begin{aligned}
 -\Delta z+f(z+v)=\eta -\xi\\
h^{1} \leq z+v\leq h^{2}\\
\int_{D}(z+v-h^{1})\eta(dx)=\int_{D}(h^2-z-v)\xi(dx)=0.
 \end{aligned}
  \right.
 \end{eqnarray}
Step1:\\
Now for $\psi \in C_{0}^{\infty}(D)$, $z^{\epsilon}$ satisfies the following integral equation:
\begin{equation}
-(\Delta z^{\epsilon},\psi)+(f(z^{\epsilon}+v),\psi)+(\frac{1}{\epsilon}(z^{\epsilon}+v-h^{2})^{+},\psi)=\int \psi(x)\eta^{\epsilon}(dx)  \label{z's sequence}
\end{equation}
i.e.
\begin{equation}\label{002}
-(z^{\epsilon},\Delta\psi)+(f(z^{\epsilon}+v),\psi)=\int \psi(x) (\eta^{\epsilon}-\xi^{\epsilon})(dx),
\end{equation}
where $\xi^{\epsilon}=\frac{(z^{\epsilon}+v-h^{2})^{+}}{\epsilon}$.
The limit of the left hand side  of (\ref{002}) exists as $\epsilon \rightarrow 0$. Therefore $\lim_{\epsilon \rightarrow 0}(\eta^{\epsilon}-\xi^{\epsilon})$ exists in the space of distributions,
i.e.
\begin{equation}
-(z,\Delta \psi)+(f(z+v),\psi)=\lim_{\epsilon\rightarrow 0}(\eta^{\epsilon}-\xi^{\epsilon},\psi) \label{limit function}
\end{equation}
Next we want to show that both $\lim_{\epsilon \rightarrow 0}\eta^{\epsilon}$ and $\lim_{\epsilon \rightarrow 0}\xi^{\epsilon}$ exist.
By Dini theorem, we know that $z^{\epsilon}(x) \rightarrow z(x)$ uniformly on compact subsets of $D$. For $\phi(x)\in C_{0}^{\infty}(D)$, denote by $K=supp(\phi)$, the compact support of $\phi$.
As $h^{1}(x)<h^{2}(x)$ in $D$, there exists $\theta_{K} >0$ such that $h^{2}(x)-h^{1}(x)\geq \theta_{K}$ on $K$ .On the other hand, there exists $\epsilon_{0}>0$, such that for $\epsilon<\epsilon_{0}$,
$|z^{\epsilon}(x)-z(x)|<\frac{\theta_{K}}{4}$ on $K$.
Let $\theta_{K}$ be chosen as above. Since
$$ supp \eta^{\epsilon} \subseteq \{x:z^{\epsilon}(x)+v(x)=h^{1}(x)\},
$$
and
$$supp \xi^{\epsilon}=\{x:z^{\epsilon}(x)+v(x)\geq h^{2}(x)\},
$$
we have  for $\epsilon \leq \epsilon_{0}$,
$$supp\eta^{\epsilon}\cap K\subseteq \{x:z(x)-\frac{\theta_{K}}{4}+v(x)\leq h^{1}(x)\}\cap K:=A_{K},
$$
and
$$supp \xi^{\epsilon}\cap K\subseteq \{x: z(x)+\frac{\theta_{K}}{4}+v(x) \geq h^{2}(x)\}\cap K:=B_{K},
$$
for $\epsilon <\epsilon_{0}.$\\
By the choice of $\theta_{K}$, we see that $A_{K}\cap B_{K}=\varnothing$. Thus, we can find $\tilde{\phi}(x)\in C_{0}^{\infty}(D)$ such that $\tilde{\phi}=\phi$ on $A_{K}$, $supp \tilde{\phi}\cap B_{K}=\varnothing$ and $supp \tilde{\phi}\cap supp \xi^{\epsilon}=\varnothing$ for $\epsilon<\epsilon_{0}$.
Hence,
$\lim_{\epsilon \rightarrow 0}(\eta^{\epsilon},\phi)=\lim_{\epsilon\rightarrow 0}(\eta^{\epsilon},\tilde{\phi})=\lim_{\epsilon\rightarrow 0}(\eta^{\epsilon}-\xi^{\epsilon},\tilde{\phi})
$
exists.
Therefore $\eta^{\epsilon}\rightarrow \eta$ in the space of distributions. Similarly, $\xi^{\epsilon}\rightarrow \xi$.
Let $\epsilon \rightarrow 0$ in equation (\ref{limit function}) to see that $(z,\eta,\xi)$ satisfies the following equation:
\begin{equation}
-(\Delta z,\psi)+(f(z+v),\psi)=\int_{D}\psi(x)(\eta-\xi)(dx).
\end{equation}
Step 2:\\
Multiplying (\ref{z's sequence}) by $\epsilon$ and letting $\epsilon \rightarrow 0$, we get
$$
0=((z+v-h^{2})^{+},\psi).
$$
This implies $z+v-h^{2}\leq 0$, i.e. $ z+v\leq h^{2}$.
Since $h^{1}\leq z^{\epsilon}+v$, we see that $h^{1}\leq z+v$.
So $h^{1}\leq z+v\leq h^{2}$.\\
Step 3:\\
Now let us show that $$ \int_{D}(z+v-h^{1})\eta(dx)=0$$ and $$\int_{D}(z+v-h^{2})\xi(dx)=0.$$

By the definition of $\xi^{\epsilon}=\frac{(z^{\epsilon}+v-h^{2})^{+}}{\epsilon},$
$\int_{D}(z^{\epsilon}+v-h^{2})\xi^{\epsilon}(dx)\geq 0$,
and the uniform convergence of $z^{\epsilon}$ on compact subsets, letting $\epsilon \rightarrow 0,$ we have $\int_{D}(z+v-h^{2})\xi(dx)\geq 0.$
Hence we must have $\int_{D}(z+v-h^{2})\xi(dx)=0.$
From the single reflecting barrier problem $(0;-\frac{(\cdot+v-h^{2})^{+}}{\epsilon};h^{1}),$ we know $\int_{D}(z^{\epsilon}+v-h^{1}) \eta^{\epsilon}(dx)=0.$ Then letting $\epsilon \downarrow 0$, we get $\int_{D}(z+v-h^{1})\eta(dx)=0.$\\
Step 4:\\
For any compact set $K\subset D$, since
$$-(\Delta z,\psi)+(f(z+v),\psi)=\int_{D}\psi(x)\eta(dx)-\int_{D}\psi(x)\xi(dx).
$$
Choose a non-negative function $\psi \in C_{0}^{\infty}(D)$ such that $\psi(x)=1$
 on $supp(\eta)\cap K$ and $\psi(x)=0$ on $supp(\xi)\cap K$,
 $$-(\Delta z,\psi)+(f(z+v),\psi)=\int_{K}\eta(dx)-0,
 $$
 So we get $\eta(K)< \infty$. Similarly, $\xi(K)<\infty$.\\

 Uniqueness:
 Let $(z,\eta,\xi)$ and $(\bar{z},\bar{\eta},\bar{\xi})$ be solutions to a double reflection problem $(0;f;h^{1},h^{2})$. We set $\Psi=z-\bar{z}$. For any $\psi \in C_{0}^{\infty}(D)$, we have
 \begin{eqnarray} \nonumber
 &&-\int_{D}\Psi(x)\Delta \psi(x) dx+\int_{D}[f(z+v)-f(\bar{z}+v)] \psi(x) dx\\ \label{z minus z bar}
   &=&\int_{D}\psi(x)\eta(dx)-\int_{D}\psi(x)\xi(dx)-\int_{D}\psi(x)\bar{\eta}(dx)+\int_{D}\psi(x)\bar{\xi}(dx). \label{distributional}
 \end{eqnarray}
 From here, following the same arguments as that in the proof of Theorem 2.2 in \cite{NT}, we can show that $z=\bar{z}$.

 Recall that
 \begin{eqnarray*}
 supp \eta, supp \bar{\eta} \subset \{x\in D:z+v=h^{1}\}=:A,\\
 supp \xi, supp \bar{\xi} \subset \{x\in D:z+v=h^{2}\}=:B.
 \end{eqnarray*}
 Because $A\cap B=\varnothing$, for any $\psi \in C_{0}^{\infty}(D)$ with $supp \psi \subset supp \eta \cup supp \bar{\eta}$, it holds that $supp \psi \cap supp \xi=\varnothing$ and $supp \psi \cap supp \bar{\xi}=\varnothing$. Applying equation (\ref{distributional}) to such a function $\psi$, we deduce that $\eta=\bar{\eta}$. Similarly $\xi=\bar{\xi}$. Then the uniqueness is proved. $\Box$

\section{Reflected SPDEs}
Recall
\begin{equation}
\left\{
\begin{aligned}
-\Delta u(x)+f(x,u(x))=\sigma(x,u(x))\dot{W}(x)+\eta(x)-\xi(x)\\
u|_{\partial D}=0\\
h^1(x)\leq u(x)\leq h^2(x)\\
\int_{D}(u(x)-h^1(x))\eta(dx)=\int_{D}(h^2(x)-u(x))\xi(dx)=0.
\end{aligned}
\right. \label{u equation}
\end{equation}

Let $G_{D}(x,y)$ be the Green function on $D$ associated to the Laplacian operator with Dirichlet boundary conditions.
We recall from \cite{BP} (or \cite{ST}) that if $k=2$ or $3$,
\begin{eqnarray*}
G_{D}(x,y)=G(x,y)-E_{x}(G(B_{\tau},y)), x,y \in D
\end{eqnarray*}
with 
\begin{eqnarray*}
G(x,y)=\frac{1}{2\pi} log|x-y|, \quad \mbox{if} \quad k=2;\\
G(x,y)=-\frac{1}{4\pi}|x-y|^{-1}, \quad \mbox{if} \quad k=3;
\end{eqnarray*}
and $B_{\tau}$ is the random variable obtained by stopping a k-dimensional Brownian motion starting at $x$ at its first exit time of $D$. For $k=1$,if $D=(0,1)$, then $G_{D}(x,y)=(x\wedge y)-xy$.\\
 
The main result of this paper is the following theorem.
\begin{theorem}
Assume that (F1), (H1) and ($\Sigma$1)  with $C_{\sigma}$ satisfying
$\exists p >1$,
\begin{equation}\label{006}
[2^{2p-1}ac_{p}Br_D^{\lambda p -k}+2^{2p-1}c_{p}(C_D)^{\frac{p}{2}}]C_{\sigma}^{p}<1,
\end{equation}
where $c_p$ and $a$ are universal constants appeared in the Burkholder's inequality, Komogorov's inequality, $r_D$
is the diameter of the domain $D$ (see(\ref{005}), (\ref{iteration of u})), $C_D=\sup_{x}\int_{D}|G_{D}(x,y)|^{2}dy$. And $B$ is the constant appeared in the estimate of the Green function $G_D$ in (4.13). 
And $\lambda$ is any number in $(0,1]$ when the dimension $k=1$; $\lambda$ is any number in $(0,1)$ when the dimension $k=2$; $\lambda$ is any number in $(0,\frac{1}{2})$ when the dimension $k=3$.\\
Then there exists a unique solution $(u,\eta,\xi)$ to the reflected SPDE Eq(\ref{elliptic equation}). Moreover, $E(||u||_{\infty})^{p}< \infty$.
\end{theorem}
\begin{proof}\\
Existence:\\
We will use successive iteration:\\
Let
\begin{equation}
v_{1}(x)=\int_{D}G_{D}(x,y)\sigma(y;0)W(dy).
\end{equation}

As in \cite{BP}, it is seen that $v_{1}(x)$ is the solution of the following SPDE:
\begin{equation}
\left\{
   \begin{aligned}
-\Delta v_{1}(x)=\sigma(x;0)\dot{W}(x)\\
v_{1}|_{\partial D}=0\\
\end{aligned}
  \right.
\end{equation}
and $v_{1}(x)\in C(\bar{D})$.\\
Denote by $(z_{1},\eta_{1},\xi_{1})$ be the unique random solution of the following reflected PDE:
\begin{equation}
\left\{
   \begin{aligned}
-\Delta z_{1}(x)+f(z_{1}+v_{1})=\eta_{1}(x)-\xi_{1}(x)\\
z_{1}|_{\partial D}=0\\
h^{1}(x) \leq z_{1}(x)+v_{1}(x) \leq h^2(x)\\
\int_{D}(z_{1}(x)+v_{1}(x)-h^1(x))\eta_{1}(dx)=\int_{D}(h^2(x)-z_{1}(x)-v_{1}(x))\xi_{1}(dx)=0.
\end{aligned}
  \right.
\end{equation}
Set $u_{1}=z_{1}+v_{1}$. Then we can easily verify that $(u_{1},\eta_{1},\xi_{1})$ is the unique solution of the following reflected SPDE:
\begin{equation}
\left\{
   \begin{aligned}
   -\Delta u_{1}(x)+f(x;u_{1})=\sigma(x;0)\dot{W}(x)+\eta_{1}(x)-\xi_{1}(x)\\
   u_{1}|_{\partial D}=0\\
   h^{1}(x) \leq u_1(x) \leq h^2(x)\\
\int_{D}(u_{1}(x)-h^1(x))\eta_{1}(dx)=\int_{D}(h^2(x)-u_{1}(x))\xi_{1}(dx)=0.
   \end{aligned}
  \right.
\end{equation}
Iterating this procedure, suppose $u_{n-1}$ has been defined. Let
\begin{equation}
v_{n}(x)=\int_{D}G_{D}(x,y)\sigma(y;u_{n-1})W(dy),
\end{equation}
and $(z_{n},\eta_{n},\xi_{n})$ be the unique random solution of the following reflected PDE:
\begin{equation}
\left\{
   \begin{aligned}
-\Delta z_{n}(x)+f(z_{n}+v_{n})=\eta_{n}(x)-\xi_{n}(x)\\
z_{n}|_{\partial D}=0\\
h^{1}(x)\leq z_{n}(x)+v_{n}(x)\leq h^{2}(x)\\
\int_{D}(z_{n}(x)+v_{n}(x)-h^1(x))\eta_{n}(dx)=\int_{D}(h^2(x)-z_{n}(x)-v_{n}(x))\xi_{n}(dx)=0.
\end{aligned}
  \right.
\end{equation}
Set $u_{n}=z_{n}+v_{n}$. Then $(u_{n},\eta_{n},\xi_{n})$ is the unique solution of the following  reflected SPDE:
\begin{equation}
\left\{
   \begin{aligned}
   -\Delta u_{n}(x)+f(x;u_{n}(x))=\sigma(x;u_{n-1}(x))\dot{W}(x)+\eta_{n}-\xi_{n}\\
   u_{n}|_{\partial D}=0\\
   h^{1}(x)\leq u_{n}(x)\leq h^{2}(x)\\
   \int_{D}(u_{n}(x)-h^1(x))\eta_{n}(dx)=\int_{D}(h^2(x)-u_{n}(x))\xi_{n}(dx)=0.
\end{aligned}
  \right.
\end{equation}
From the proof of Lemma 3.4 (also Lemma 3.1 in \cite{NT}), we have
\begin{equation}
||z_{n}-z_{n-1}||_{\infty}\leq ||v_{n}-v_{n-1}||_{\infty}, \label{z iteration}
\end{equation}
hence
\begin{equation}
||u_{n}-u_{n-1}||_{\infty}
   \leq 2||v_{n}-v_{n-1}||_{\infty}.           \label{u's difference}
\end{equation}
Namely,
\begin{eqnarray}\label{003}
&&(||u_{n}-u_{n-1}||_{\infty})\nonumber\\
   &&\leq 2\sup_{x\in D}|\int_{D}G_{D}(x,y)(\sigma(y;u_{n-1})-\sigma(y;u_{n-2}))W(dy)|.
\end{eqnarray}
Set
\begin{eqnarray*}
I(x)=\int_{D}G_{D}(x,y)(\sigma(y;u_{n-1})-\sigma(y;u_{n-2}))W(dy).
\end{eqnarray*}
Then $\forall p\geq 1$,
\begin{eqnarray*}
&&  E[|I(x)-I(y)|^{p}]\\
&&=  E|\int_{D}(G_{D}(x,z)-G_{D}(y,z))(\sigma(z;u_{n-1})-\sigma(z;u_{n-2}))W(dz)|^{p}\\
&&\leq c_{p}E[\int_{D}|G_{D}(x,z)-G_{D}(y,z)|^{2}\cdot|\sigma(z;u_{n-1})-\sigma(z;u_{n-2})|^{2} dz]^{\frac{p}{2}}\\
&&\leq c_{p}E||\sigma(z;u_{n-1})-\sigma(z;u_{n-2})||_{\infty}^{p}
[\int_{D}|G_{D}(x,z)-G_{D}({y,z})|^2dz]^{\frac{p}{2}},
\end{eqnarray*}
where $c_p$ is a Burkholder constant only related to $p$.\\
Similarly as the proof of Theorem 3.3 in \cite{ST}, we have
\begin{eqnarray}
||G_{D}(x,z)-G_{D}(y,z)||_{L^{2}(D)}^{2}\leq B|x-y|^{2\lambda },
\end{eqnarray}
where $\lambda=1$ when $k=1$, $\lambda$ is arbitrarily close to $1$ when $k=2$, and
$\lambda$ is arbitrarily close to $\frac{1}{2}$ when $k=3$.
Then,
\begin{eqnarray}
  E|I(x)-I(y)|^{p}
   \leq
  E||\sigma(z;u_{n-1})-\sigma(z;u_{n-2})||_{\infty}^{p}c_{p}B|x-y|^{\lambda p}.
\end{eqnarray}
We next show that $u_n$ converges uniformly on  $D$.
 Let $K \subset D$ be any compact subset of $D$.  $\forall x,y \in K,$ from Kolmogrov lemma (Lemma 3.1 in \cite{DP} ), we deduce that for $\forall p> \frac{k}{\lambda}$,
\begin{eqnarray}
|I(x)-I(y)|^{p}\leq (N(w))^{p}|x-y|^{\lambda p-k}(log(\frac{\gamma}{|x-y|}))^{2},\\
E(N^{p})\leq a c_{p}B E||\sigma(z;u_{n-1})-\sigma(z;u_{n-2})||_{\infty}^{p},
\end{eqnarray}
where $a$ is a universal constant independent of $K$.
Choosing $y=x_{0}\in K$, we see that
\begin{eqnarray}
E[\sup_{x\in K}|I(x)|^{p}]
& \leq&2^{p-1} E[\sup_{x\in K}|I(x)-I(x_{0})|^{p}]+2^{p-1}E|I(x_{0})|^{p}\nonumber \\
&\leq&2^{p-1} E(N^{p})r_{D}^{\lambda p -k}+2^{p-1} E|I(x_{0})|^{p}\nonumber \\
&\leq& 2^{p-1}ac_{p}Br_D^{\lambda p -k}E||\sigma(z;u_{n-1})-\sigma(z;u_{n-2})||_{\infty}^{p}\nonumber\\
&&+2^{p-1}E|I(x_{0})|^{p},
\end{eqnarray}
where $r_D=sup_{x,y\in D}|x-y|$ is the diameter of $D$.
Furthermore,
\begin{eqnarray}
E|I(x_{0})|^{p}
&=& E|\int_{D}G_{D}(x_{0},y)(\sigma(y;u_{n-1})-\sigma(y;u_{n-2}))W(dy)|^{p}\nonumber\\
&\leq& c_{p}E[\int_{D}|G_{D}(x_{0},y)|^{2}\cdot|\sigma(y;u_{n-1})-\sigma(y;u_{n-2})|^{2} dy]^{\frac{p}{2}}\nonumber\\
&\leq& c_{p}(C_D)^{\frac{p}{2}}E||\sigma(z;u_{n-1})-\sigma(z;u_{n-2})||_{\infty}^{p},
\end{eqnarray}
where $C_D=\sup_x\int_{D}|G_{D}(x,y)|^{2}dy <\infty$.
So we have
\begin{eqnarray}\label{004}
&&E[\sup_{x\in K}|I(x)|^{p}]\nonumber\\
&\leq& 2^{p-1}ac_{p}Br_D^{\lambda p -k}E||\sigma(z;u_{n-1})-\sigma(z;u_{n-2})||_{\infty}^{p}\nonumber\\
&&+2^{p-1}c_{p}(C_D)^{\frac{p}{2}}E||\sigma(z;u_{n-1})-\sigma(z;u_{n-2})||_{\infty}^{p}.
\end{eqnarray}
Since the constants on the right side of  (\ref{004}) are independent of the compact subset $K$, by Fatou's Lemma we deduce that
\begin{eqnarray}\label{005}
&&E[\sup_{x\in D}|I(x)|^{p}]\nonumber\\
&\leq& 2^{p-1}ac_{p}B r_D^{\lambda p -k}E||\sigma(z;u_{n-1})-\sigma(z;u_{n-2})||_{\infty}^{p}\nonumber\\
&&+2^{p-1}c_{p}(C_D)^{\frac{p}{2}}E||\sigma(z;u_{n-1})-\sigma(z;u_{n-2})||_{\infty}^{p}.
\end{eqnarray}
Now it follows from (\ref{003}) and (\ref{005}) that
\begin{eqnarray} \nonumber
&&E(||u_{n}-u_{n-1}||_{\infty})^{p}\nonumber\\
&\leq& [2^{2p-1}ac_{p}B r_D^{\lambda p -k}+2^{2p-1}c_{p}(C_D)^{\frac{p}{2}}]\nonumber\\
&&\quad\quad \times E||\sigma(z;u_{n-1})-\sigma(z;u_{n-2})||_{\infty}^{p}\nonumber\\
&\leq&[2^{2p-1}ac_{p}B r_D^{\lambda p -k}+2^{2p-1}c_{p}(C_D)^{\frac{p}{2}}]C_{\sigma}^{p}\nonumber\\
&&\quad\quad \times E||u_{n-1}-u_{n-2}||_{\infty}^{p}\nonumber\\
&\leq& ...\nonumber\\
&\leq& \left ( [2^{2p-1}ac_{p}B r_D^{\lambda p -k}+2^{2p-1}c_{p}(C_D)^{\frac{p}{2}}]C_{\sigma}^{p}\right )^{n-1} E||u_{1}-u_{0}||_{\infty}^{p}       \label{iteration of u}
\end{eqnarray}

Since
  $$[2^{2p-1}ac_{p}Br_D^{\lambda p -k}+2^{2p-1}c_{p}(C_D)^{\frac{p}{2}}]C_{\sigma}^{p}<1,$$
we obtain from (\ref{iteration of u}) that  for any $m\geq n\geq 1$,
\begin{eqnarray}\nonumber
E(||u_{m}-u_{n}||_{\infty})^{p}
            \rightarrow 0,
\end{eqnarray}
as $n,m \rightarrow \infty$.\\
Hence, there exists a continuous random field $u(\cdot ) \in C(D)$, such that
\begin{eqnarray}
E((||u||_{\infty})^{p})< \infty,
\end{eqnarray}
and
\begin{eqnarray}
\lim_{n\rightarrow \infty}E(||u_{n}-u||_{\infty})^{p}=0.
\end{eqnarray}
Next we will show that $u$ is a solution of Eq(\ref{u equation}).\\
Set
\begin{eqnarray}
v(x)=\int_{D}G_{D}(x,y)\sigma(y;u)W(dy).
\end{eqnarray}
As the proof of (\ref{005}), we have
\begin{eqnarray}
\lim_{n\rightarrow \infty}E||v_{n}-v||_{\infty}^{p}=\lim_{n\rightarrow \infty}E||u_{n-1}-u||_{\infty}^{p}=0.
\end{eqnarray}
From the inequality (\ref{z iteration}), there exists a continuous random field $z(x)$ on D such that\\
 $\lim_{n\rightarrow \infty}E||z_{n}-z||_{\infty}^{p}=0$. So $z_{n}$ converges to $z$ uniformly on D. Similar to the proof of Theorem 3.1, we can show that $\eta(dx)=\lim_{n\rightarrow \infty} \eta_{n}(dx)$, $\xi(dx)=\lim_{n\rightarrow \infty} \xi_{n}(dx)$ exist almost surely and $(z,\eta,\xi)$ is the solution of equation (\ref{deterministic eq1}) with the above given $v$. Put $u(x)=z(x)+v(x)$. It is easy to verify $(u,\eta,\xi)$ is a solution to the SPDE(\ref{u equation}) with two reflecting walls.\\

Uniqueness:\\
Let $(u_{1},\eta_{1},\xi_{1})$ and $(u_{2},\eta_{2},\xi_{2})$ be two solutions of Eq(\ref{u equation}). Set
\begin{eqnarray}
v_{1}(x)=\int_{D} G_{D}(x,y)\sigma(y;u_{1})W(dy),\\
v_{2}(x)=\int_{D}G_{D}(x,y)\sigma(y;u_{2})W(dy),
\end{eqnarray}
and
 $z_{1}=u_{1}-v_{1}$ and $z_{2}=u_{2}-v_{2}$.
Then  $z_{1}$, $z_{2}$ are solutions of the following reflected random PDEs:
\begin{eqnarray}
\left\{
   \begin{aligned}
-\Delta z_{1}(x)+f(z_{1}+v_{1})=\eta_{1}(x)-\xi_{1}(x)\\
z_{1}|_{\partial D}=0\\
h^{1}(x)\leq z_{1}(x)+v_{1}(x)\leq h^{2}(x)\\
\int_{D}(z_{1}(x)+v_{1}(x)-h^1(x))\eta_{1}(dx)=\int_{D}(h^2(x)-z_{1}(x)-v_{1}(x))\xi_{1}(dx)=0,
\end{aligned}
  \right.
\end{eqnarray}
\begin{eqnarray}
\left\{
   \begin{aligned}
-\Delta z_{2}(x)+f(z_{2}+v_{2})=\eta_{2}(x)-\xi_{2}(x)\\
z_{2}|_{\partial D}=0\\
h^{1}(x)\leq z_{2}(x)+v_{2}(x)\leq h^{2}\\
\int_{D}(z_{2}(x)+v_{2}(x)-h^1(x))\eta_{2}(dx)=\int_{D}(h^2(x)-z_{2}(x)-v_{2}(x))\xi_{2}(dx)=0,
\end{aligned}
  \right.
\end{eqnarray}
Similar to the inequality (\ref{z iteration}), we have
\begin{eqnarray}
||z_{1}-z_{2}||_{\infty} \leq ||v_{1}-v_{2}||_{\infty}.
\end{eqnarray}
Hence,
\begin{eqnarray}\label{006}
||u_{1}-u_{2}||_{\infty}^{p} \nonumber
&\leq& 2^{p}||v_{1}-v_{2}||_{\infty}^{p}\\
&\leq& 2^{p}(\sup_{x\in D}|\int_{D}G_{D}(x,y)(\sigma(y;u_{1})-\sigma(y;u_{2}))W(dy)|^{p})
\end{eqnarray}
As the proof of  (\ref{iteration of u}), we deduce from (\ref{006}) that
\begin{eqnarray}
E||u_{1}-u_{2}||_{\infty}^{p}\nonumber
&\leq &[2^{2p-1}ac_{p}C(p)r_D^{\lambda p -k}+2^{2p-1}c_{p}(C_D)^{\frac{p}{2}}]C_{\sigma}^{p}\nonumber\\
&&\quad \times  E||u_{1}-u_{2}||_{\infty}^{p}.
\end{eqnarray}
As
$$[2^{2p-1}ac_{p}C(p)r_D^{\lambda p -k}+2^{2p-1}c_{p}(C_D)^{\frac{p}{2}}]C_{\sigma}^{p}<1,$$
it follows that
\begin{eqnarray}
u_{1}=u_{2} \ \ a.s.
\end{eqnarray}
On the other hand, for $\phi \in C_{0}^{\infty}(D)$,
\begin{eqnarray}\nonumber
&& - (u_{1}(x)-u_{2}(x),\Delta \phi(x))+(f(x,u_{1}(x))-f(x;u_{2}(x)),\phi(x))\\ \nonumber
&& =\int_{D}[\sigma(x;u_{1}(x))-\sigma(x;u_{2}(x))]\phi(x)W(dx)\\
&& \ \ +\int_{D}\phi(x)(\eta_{1}(dx)-\eta_{2}(dx))-\int_{D}\phi(x)(\xi_{1}(dx)-\xi_{2}(dx)). \label{two solutions' difference}
\end{eqnarray}
Therefore we have
\begin{eqnarray}
\int_{D}\phi(x)(\eta_{1}(dx)-\eta_{2}(dx))-\int_{D}\phi(x)(\xi_{1}(dx)-\xi_{2}(dx))=0.\label{eta and xi uniqueness}
\end{eqnarray}
Recall that
\begin{eqnarray*}
supp \eta_{1}, supp \eta_{2} \subset \{x\in D:u_{1}(x)=h^{1}(x)\}=:A\\
supp \xi_{1}, supp \xi_{2} \subset \{x\in D:u_{1}(x)=h^{2}(x)\}=:B.
\end{eqnarray*}
Because $A\bigcap B=\emptyset $, for any $\phi \in C_{0}^{\infty}(D)$ with $supp\phi \subset (supp \eta_{1}\bigcup supp \eta_{2})$, it holds that $supp \phi \bigcap supp \xi_{1}=\emptyset$ and $supp \phi \bigcap supp \xi_{2}=\emptyset$.
Applying equation(\ref{eta and xi uniqueness}) to such a function $\phi$, we deduce that $\eta_{1}=\eta_{2}$. Similarly, $\xi_{1}=\xi_{2}$. Then the uniqueness is proved. $\Box$
\end{proof}

\end{document}